\documentclass[12pt]{article}

\usepackage{amssymb}
\usepackage{amsfonts}
\usepackage{theorem}
\usepackage{bbm}
\usepackage{graphicx}
\usepackage[leqno]{amsmath}
\usepackage{enumerate}
\usepackage{indentfirst}
\usepackage{color}

\newtheorem{thm}{Theorem}

\newtheorem{remark}{Remark}
\newtheorem{lmm}{Lemma}[section]

\newenvironment{proof}[1][Proof]{\textbf{#1.} }{\hfill $\square$}

\newcommand{\R}{\mathbbm{R}}

\newcommand{\al}{\alpha}
\newcommand{\eps}{\varepsilon}

\newcommand{\mE}{\mathbf{E}}

\newcommand{\cL}{\mathcal{L}}

\newcommand{\trace}{\mbox{Tr}}

\newcommand{\aeff}{a^{{\rm eff}}}

\title{Higher order homogenization for random non-autonomous parabolic operators}
\author{Marina Kleptsyna\thanks{Le Mans Universit\'e, Laboratoire Manceau de Math\'ematiques, Avenue Olivier Messiaen, 72085 Le Mans, Cedex 9, France.
e-mail: {\tt marina.kleptsyna@univ-lemans.fr, \ alexandre.popier@univ-lemans.fr }}
\,\,\, Andrey Piatnitski \thanks{The Arctic University of Norway, campus Narvik,  P.O.Box 385,
8505 Narvik, Norway\ \ \  and \ \ \   Institute for Information Transmission Problems of RAS, 19, Bolshoy Karetny per., Moscow 127051,
Russia.\ e-mail: {\tt api009@uit.no}}
\,\,\, and Alexandre Popier\footnotemark[3]
}

\date{\today}

\begin{document}
\maketitle

\begin{abstract}
We consider Cauchy problem for a divergence form second order parabolic operator with rapidly oscillating
coefficients that are periodic in spatial variables and random stationary ergodic in time. As was proved
in \cite{JKO_1} and \cite{KP_1995} in this case the homogenized operator is deterministic. The paper focuses on
the diffusion approximation of solutions in the case of non-diffusive scaling, when the oscillation in spatial variables is faster than that in temporal variable. Our goal is to study the asymptotic behaviour of the normalized difference between solutions of the original
and the homogenized problems.
\end{abstract}

\section{Introduction}\label{s_intro}
In this work we consider the asymptotic behaviour of solutions to the following Cauchy problem
\begin{equation}\label{ori_cauch}
\begin{array}{c}
\displaystyle
\frac\partial{\partial t}u^\eps=\mathrm{div}\Big[a\Big(\frac x\eps,\frac t{\eps^\alpha}\Big)\nabla u^\eps\Big]\qquad
\hbox{in }\mathbb R^d\times(0,T]\\[4mm]
u^\eps(x,0)=\varphi(x).
\end{array}
\end{equation}
Here $\eps$ is a small positive parameter that tends to zero, $\alpha$ satisfies the inequality $0<\alpha<2$,
$a(z,s)$ is a positive definite matrix whose entries are periodic in $z$ variable and random stationary ergodic in $s$.

It is known (see \cite{JKO_1, KP_1995}) that this problem admits homogenization and that the homogenized operator
is deterministic and has constant coefficients.  The homogenized Cauchy problem takes the form
\begin{equation}\label{eff_cauch}
\begin{array}{c}
\displaystyle
\frac\partial{\partial t}u^0=\mathrm{div}(\aeff \nabla u^0)\\[3mm]
u^0(x,0)=\varphi(x).
\end{array}
\end{equation}
The formula for the effective matrix $\aeff$ is given in \eqref{def_effema} in Section 2 (see also \cite{KP_1995}).

The goal of this paper is to study the limit behaviour of the difference $u^\eps-u^0$, as $\eps\to0$.

In the existing literature there is a number of works devoted to homogenization of random parabolic problems.
The results obtained in \cite{Ko78} and \cite{PV80} for random divergence form elliptic operators also apply
to the parabolic case. In the presence of large lower order terms the limit dynamics might remain random
and show diffusive or even more complicated behaviour.  The papers \cite{CKP_2001}, \cite{PP}, \cite{KP_3}
focus on the case of time dependent parabolic operators  with periodic in spatial variables and random
in time coefficients. The fully random case has been studied in \cite{PP_1}, \cite{Ba_1}, \cite{Ba_2},
\cite{HPP}.

One of the important aspects of homogenization theory is  estimating the rate of convergence.
For random operators the first estimates have been obtained in \cite{Yu80}. Further important progress in this direction
was achieved in the recent works \cite{GO12}, \cite{GM12}.

Problem (\ref{ori_cauch}) in the case of diffusive scaling $\alpha=2$ was studied in our previous work \cite{KPP_2015}.
It was shown that, under proper mixing conditions, the difference $u^\eps-u^0$ is of order $\eps$, and that
the normalized difference $\eps^{-1}(u^\eps-u^0)$ after subtracting an appropriate corrector, converges in law
to a solution of some limit SPDE.

In the present paper we consider the case $0<\alpha<2$.    In other words, bearing in mind the diffusive scaling,
we assume that the oscillation in spatial variables is faster than that in time. In this case the principal part of
the asymptotics of $u^\eps-u^0$ consists of a finite number of correctors, the oscillating part of each of them being
a solution of an elliptic PDE with periodic in spatial variable coefficients. The number of correctors increases as
$\alpha$ approaches~$2$. After subtracting these correctors,
the resulting expression divided by $\eps^{\alpha/2}$ converges in law to a solution of the limit SPDE.

In contrast with the diffusive scaling, for $\alpha<2$  the interplay between the scalings in spatial variables and time
and the necessity  to construct higher order correctors results in additional regularity assumptions on the coefficients.
Indeed, each corrector is introduced as a solution of some elliptic equation in which time is a parameter, thus
this corrector has the same regularity in time as the coefficients of the equation. When we construct the next term of the expansion,
this corrector is differentiated in time. This differentiation reduces the regularity.
The result mentioned in the previous paragraph holds if the coefficients $a^{ij}(z,s)$ in \eqref{ori_cauch} are smooth enough
functions.

We also consider in the paper the special case of diffusive dependence on time. Namely, we assume in this case
that $a(z,s)={\rm a}(z,\xi_s)$, where $\xi_\cdot$ is a stationary diffusion process in $\mathbb R^n$ and ${\rm a}(z,y)$
is a  periodic in $z$ smooth deterministic function. It should be emphasized that in the said diffusive case
Theorem \ref{t_1} does not apply because the coefficients $a^{ij}$ do not possess the required regularity in time. This lack of regularity leads to additional difficulties in treating the diffusive case. As was shown in our previous work \cite{KPP22}, the statement of Theorem \ref{t_1} remains valid if $\alpha <1$.  Also, for $1\leq \alpha < 2$ in dimension one the issues can be addressed using the equation satisfied by the potential of the discrepancy. This technic fails to work in dimension higher than 1. Here we treat the case $\alpha=1$ in any dimension.

\bigskip

The paper is organized as follows.
\begin{itemize}
\item In Section \ref{s_1} we introduce the studied problem and provide all the assumptions. Then we formulate the main result (Theorem \ref{t_1}) of the paper concerning the smooth case for $\alpha < 2$.

Section \ref{s_sec3} focuses on the proof of Theorem \ref{t_1}. At the beginning we consider a number of auxiliary problems and define the higher order terms of the asymptotics of solution.
\item In Section \ref{s_sec4} we consider the special case of diffusive dependence on time for $\alpha \leq 1$. We extend to the dimension $d$ the result of \cite{KPP22} in Theorem \ref{t_2diff}.

\end{itemize}

 \section{The smooth case}\label{s_1}

In this section we provide all the assumptions on the data of problem \eqref{ori_cauch}, introduce some notations and formulate the main result.

For the studied Cauchy problem \eqref{ori_cauch}, where $\eps$ is a small positive parameter, we assume that the following conditions hold true:
\begin{itemize}
\item [\bf a1.]the matrix $a(z,s)=\{a^{ij}(z,s)\}\big._{i,j=1}^d$ is  symmetric and satisfies uniform ellipticity conditions
$$
\lambda |\zeta|^2\leq a(z,s)\zeta\cdot\zeta\leq \lambda^{-1} |\zeta|^2,\quad \zeta\in\mathbb R^d,\ \ \lambda>0;
$$
\item [\bf a2.] $\varphi\in C_0^\infty(\mathbb R^d)$. In fact, this condition can be essentially relaxed, see Remark \ref{r_regu_ini}.
\end{itemize}

In the first setting it is assumed that the coefficients of matrix $a$ are smooth functions
that have good mixing properties in time variable. The smoothness is important because our approach relies on auxiliary elliptic
equations that depend on time as a parameter, and we have to differentiate these equations w.r.t. time.

\medskip
In the case of smooth coefficients our assumptions read:
\begin{itemize}
\item [\bf h1.]
The coefficients $a^{ij}(z,s)$ are periodic in $z$ with the period $[0,1]^d$ and random stationary ergodic in $s$.
Given a probability space $(\Omega,\mathcal{F},{\bf P})$ with an ergodic dynamical system $\tau_s$, we assume that
$a^{ij}(z,s,\omega)=\mathtt{a}^{ij}(z,\tau_s\omega)$, where $\{\mathtt{a}^{ij}(z,\omega)\}\big._{i,j=1}^d$ is a collection of random periodic in $z$ functions that satisfy the above uniform ellipticity conditions.
\item [\bf h2.] The realizations $a^{ij}(z,s)$ are  smooth. For any  $N\geq 1$ and $k\geq 1$ there exist $C_{N,k}$
such that
$$
{\bf E}\,\|a^{ij}\|^k_{C^N(\mathbb T^d\times[0,T])}\leq C_{N,k};
$$
here and in what follows we identify periodic functions with functions on the torus $\mathbb T^d$, ${\bf E}$ stands for the expectation.
\item [\bf h3.] Mixing condition. The strong mixing coefficient $\gamma(r)$ of $a(z,\cdot)$ satisfies the inequality
$$
\int_0^\infty (\gamma(r))^{1/2}dr<\infty.
$$
\end{itemize}
We say that Condition {\bf (H)} holds if {\bf a1}, {\bf a2} and {\bf h1} -- {\bf h3} are fulfilled.

For the reader's convenience we recall here the definition of strong mixing coefficient.
Let $\mathcal{F}_{\leq s}$ and $\mathcal{F}_{\geq s}$ be the $\sigma$-algebras generated by
$\{a(z,t)\,:\,z\in\mathbb T^d, t\leq s\}$ and $\{a(z,t)\,:\,z\in\mathbb T^d, t\geq s\}$, respectively.
We set
$$
\gamma(r)=\sup \big| \mathbf{P}(A\cap B)-\mathbf{P}(A)\mathbf{P}(B)\big|,
$$
where the supremum is taken over all $A\in \mathcal{F}_{\leq 0}$ and $B\in \mathcal{F}_{\geq r}$.

\subsection{Homogenized problem and first corrector}

According to \cite{KP_1995}, under {\bf (H)}, %or {\bf (C)},
the sequence $u^\eps$ converges in probability, as $\eps\to0$, to a solution $u^0$ of problem \eqref{eff_cauch}. For the reader convenience we provide here the definition of the effective matrix $\aeff$. %If {\bf (H)} holds,
We solve the following auxiliary problem
\begin{equation}\label{aux_0}
\mathrm{div}\big(a(z,s,\omega)\nabla \chi^0(z,s,\omega)\big)=-\mathrm{div}\, a(z,s,\omega),\quad z\in\mathbb T^d;
\end{equation}
 here $s$ and $\omega$ are parameters, and $\chi^0$ is an unknown vector function: $ \chi^0=(\chi^{0,1},\ldots,\chi^{0,d})$.
 In what follows {\it we usually do not indicate explicitly the dependence of $\omega$}.
Due to ellipticity of the matrix $a$  equation \eqref{aux_0} has a unique, up to an additive constant vector, periodic solution,
$ \chi^0\in \big(L^\infty(\mathbb T^d)\cap H^1(\mathbb T^d)\big)^d $.
 This constant vector is chosen in such a way that
 \begin{equation}\label{norm_cond0}
 \int_{\mathbb T^d}\chi^0(z,s)\,dz=0\qquad\hbox{for all } s \ \hbox{and } \omega.
 \end{equation}
Then we define the effective matrix $\aeff$ by
\begin{equation}\label{def_effema}
\aeff={\bf E}\int_{\mathbb T^d}\big({\bf I}+a(z,s)\big)\nabla\chi^0(z,s)\,dz,
\end{equation}
where ${\bf I}$ stands for the unit matrix, and $\{\nabla\chi^0(z,s)\}^{ij}=\frac{\partial}{\partial z_i}\chi^{0,j}$.

It is known that the matrix $\aeff$ is positive definite (see, for instance, \cite{KP_1995}). Therefore, problem (\ref{eff_cauch}) is well posed, and function $u^0$ is uniquely defined. Under assumption {\bf a2} the function $u^0$ is smooth and satisfies the estimates
\begin{equation}\label{est_uzero}
\Big|(1+|x|)^ N\frac{\partial^{\bf k}u^0(x,t)}{\partial t^{k_0}\partial x_1^{k_1}\ldots\partial x_d^{k_d}}\Big|\leq C_{N,{\bf k}}
\end{equation}
for all $N>0$ and all multi index ${\bf k}=(k_0,k_1,\ldots, \, k_d)$, $k_i\geq 0$.

\subsection{Main result for smooth coefficients with good mixing properties}
%--------------

Here we assume that condition {\bf (H)} holds. In order to formulate the main results we need a number of auxiliary functions and quantities. For $j=1,2,\ldots, J^0$ with $J^0=\lfloor\frac\alpha{2(2-\alpha)}\rfloor+1$, the higher order correctors are
 introduced as periodic solutions to the equations
 \begin{equation}\label{aux_j}
  \mathrm{div}\big(a(z,s)\nabla \chi^j(z,s)\big)=\partial_s\chi^{j-1}(z,s),
 \end{equation}
 where $\lfloor\cdot\rfloor$ stands for the integer part. Due to \eqref{norm_cond0} for $j=1$ this equation
 is solvable in the space of periodic in $z$ functions. A solution $\chi^1$ is uniquely defined up to an additive constant vector.
 Choosing the constant vector in a proper way yields
 $$
  \int_{\mathbb T^d}\chi^1(z,s)\,dz=0\qquad\hbox{for all } s \ \hbox{and } \omega
 $$
and thus the solvability of the equation for $\chi^2$. Iterating this procedure, we define all $\chi^j$, $j=1,2,\ldots, J^0$.

Next, we introduce the functions  $u^j=u^j(x,t)$, $j=1,\ldots, J^0$. They solve the following problems:
\begin{equation}\label{eff_cauch_j}
\begin{array}{c}
\displaystyle
\frac\partial{\partial t}u^j=\mathrm{div}(\aeff\nabla u^j)+\sum\limits_{k=1}^{j}
\{a^{k,{\rm eff}}\}^{im}\frac{\partial^2}{\partial x_i\partial x_m}u^{j-k}
\\[3mm]
u^j(x,0)=0
\end{array}
\end{equation}
with
\begin{equation}\label{eq:k_effec_matrix}
a^{k,{\rm eff}}={\bf E}\int_{\mathbb T^d}a(z,s)\nabla\chi^k(z,s)\,dz;
\end{equation}
here and later on we assume summation from $1$ to $d$ over repeated indices.

To characterize the diffusive term in the limit equation we introduce the matrix
$$
\Xi(s)=\int_{\mathbb T^d}\Big[\big( a(z,s)+a(z,s)\nabla\chi^0(z,s)\big)-
{\bf E}\big\{ a(z,s)+a(z,s)\nabla\chi^0(z,s)\big\}\Big]dz.
$$
By construction the matrix function $\Xi$ is stationary and its entries satisfy condition {\bf h3} (mixing condition).
Denote
$$
\Lambda=\frac12\int_0^\infty {\bf E}\Big(\Xi(s)\otimes\Xi(0)+\Xi(0)\otimes\Xi(s)\Big)\,ds,\qquad
\Lambda=\{\Lambda^{ijkl}\},
$$
where $(\Xi(s)\otimes\Xi(0))^{ijkl}=\Xi^{ij}(s)\Xi^{kl}(0)$. Under condition {\bf h3} the integral on the right-hand side converges.

The first main result of this paper is
\begin{thm}\label{t_1}
Let Condition {\bf (H)} be fulfilled. Then the functions
$$
U^\eps=\eps^{-\alpha/2}\Big(u^\eps-u^0-\sum\limits_{j=1}^{J_0}\eps^{j(2-\alpha)}u^j\Big)
$$
 converge in law, as $\eps\to0$, in $L^2(\mathbb R^d\times(0,T))$ to the unique solution of
 the following SPDE
 \begin{equation}\label{eff_spde}
\begin{array}{c}
\displaystyle
dv^0=\mathrm{div}(\aeff\nabla v^0)\,dt+(\Lambda^{1/2})^{ijkl}\frac{\partial^2}{\partial x_i\partial x_j}u^0\,dW_t^{kl}
\\[3mm]
v^0(x,0)=0;
\end{array}
\end{equation}
where $W_\cdot=\{W_\cdot^{ij}\}$ is the standard $d^2$-dimensional Brownian motion.
\end{thm}

\begin{remark}\label{r_regu_ini}
{\rm The regularity assumption on $\varphi$ given in condition {\bf a2} can be weakened. Namely, the statement of Theorem \ref{t_1}
holds if $\varphi$ is $J^0+1$ times continuously differentiable and the corresponding partial derivatives decay at infinity
sufficiently fast.}
\end{remark}

The scheme of the proof is the following. We write down the following ansatz
$$
V^\eps(x,t)=\eps^{-\alpha/2}\Big\{u^\eps(x,t)- \sum\limits_{k=0}^{J^0}\eps^{k\delta}\Big(u^k(x,t)+\sum\limits_{j=0}^{J^0-k}\eps^{(j\delta+1)}\chi^j\Big(\frac x\eps,\frac t{\eps^\alpha}\Big)\nabla u^k(x,t)\Big)\Big\},
$$
here and in what follows the symbol $\delta$ stands for $2-\alpha$. %In the diffusive case,
%$$\chi^j\Big(\frac x\eps,\frac t{\eps^\alpha}\Big) = \chi^j\Big(\frac x\eps,\xi_{\frac t{\eps^\alpha}}\Big).$$
Then we substitute $V^\eps$ for $u^\eps$ in \eqref{ori_cauch} and we obtain for $V^\eps$ a PDE with random coefficients when {\bf (H)} is in force.
%\begin{itemize}
%\item a PDE with random coefficients when {\bf (H)} is in force;
%\item a SPDE in the diffusive case.
%\end{itemize}
We prove that $V^\eps$ converges in law in the suitable functional space to the solution of (\ref{eff_spde}). %In the case of  smooth coefficients
We combine the definition of correctors,  formula \eqref{eff_cauch_j} and the Cental Limit Theorem for stationary mixing processes. After some manipulations this yields the desired convergence (see Section \ref{s_sec3}). The rest of this section concerns the proof of this result.

%
%\subsection{Scheme of the proofs}
%\label{s_scheme}
%---------------------
%
%In both cases the beginning of the proofs is the same. We write down the following ansatz
%$$
%V^\eps(x,t)=\eps^{-\alpha/2}\Big\{u^\eps(x,t)- \sum\limits_{k=0}^{J^0}\eps^{k\delta}\Big(u^k(x,t)+\sum\limits_{j=0}^{J^0-k}\eps^{(j\delta+1)}\chi^j\Big(\frac x\eps,\frac t{\eps^\alpha}\Big)\nabla u^k(x,t)\Big)\Big\},
%$$
%here and in what follows the symbol $\delta$ stands for $2-\alpha$. In the diffusive case,
%$$\chi^j\Big(\frac x\eps,\frac t{\eps^\alpha}\Big) = \chi^j\Big(\frac x\eps,\xi_{\frac t{\eps^\alpha}}\Big).$$
%Then we substitute $V^\eps$ for $u^\eps$ in \eqref{ori_cauch} and we obtain for $V^\eps$
%\begin{itemize}
%\item a PDE with random coefficients when {\bf (H)} is in force;
%\item a SPDE in the diffusive case.
%\end{itemize}
%We prove that $V^\eps$ converges in law in the suitable functional space to the solution of (\ref{eff_spde}). In the case of  smooth coefficients
% we combine the definition of correctors,  formula \eqref{eff_cauch_j} and the Cental Limit Theorem for stationary mixing processes. After some manipulations this yields the desired convergence (see Section \ref{s_sec3}). \\

\subsection{Proof of Theorem 1}\label{s_sec3}
%-----------------

\subsubsection*{Auxiliary problems.}%\label{s_sec2}

We begin by considering problem (\ref{aux_0}). This equation has a unique up to an additive constant vector periodic solution.
Since $\chi^0(\cdot,s)$ only depends on $a(\cdot,s)$, the solution with zero average is stationary and the strong mixing
coefficient of the pair $(a(\cdot,s),\chi^0(\cdot,s))$
coincides with that for $a(\cdot,s)$. The same statement is valid for any finite collection  $(a(\cdot,s),\chi^0(\cdot,s),,\chi^1(\cdot,s),\ldots )$.
By the classical elliptic estimates, under our standing assumptions
we have
\begin{equation}\label{est_chi0}
\|\chi^0\|_{L^\infty(\mathbb T^d\times[0,T])}\leq C,\qquad
\mathbf{E}\|\chi^0\|^N_{C^k(\mathbb T^d\times[0,T])}\leq C_{k,N}.
\end{equation}
Indeed, multiplying equation \eqref{aux_0}   by  $\chi^0$, using the Schwartz and Poincar\'e inequalities and considering \eqref{norm_cond0},
we conclude that $\|\chi^0(\cdot,s)\|_{H^1(\mathbb T^d)}\leq C$ for all $s\in\mathbb R$.  The first estimate in \eqref{est_chi0} then follows from
 \cite[Theorem 8.4]{GT}. The second estimate follows from the Schauder estimates, see \cite[Chapter 6]{GT}

By the similar arguments, the solutions $\chi^j$ of equations \eqref{aux_j} are stationary, satisfy strong mixing condition with the same
coefficient $\gamma(r)$, and the following estimates hold: for any $N\geq 1$ and $k\geq 0$
\begin{equation}\label{est_chij}
\mathbf{E}\|\chi^j\|^N_{C^k(\mathbb T^d\times[0,T])}\leq C_{k,N}, \quad j=0,\,1,\ldots, J_0.
\end{equation}

The solutions $\chi^j$ defined by \eqref{aux_j_dif} satisfy the same estimate: for any $N>0$ there exists $C_N$ such that
$$
\|\chi^j\|_{C^N(\mathbb T^d\times\mathbb R^n)}\leq C_N.
$$

Solutions $u^0$ and $u^j$ of problems \eqref{eff_cauch}, \eqref{eff_cauch_j} and \eqref{eff_cauch_j_diff} are smooth functions. Moreover, for any ${\bf k}=
(k_0,k_1,\ldots,k_d)$ and $N>0$ there exists a constant $C_{{\bf k},N}$ such that
\begin{equation}\label{est_uuj}
|D^{\bf k}u^j| \leq C_{{\bf k},N}(1+|x|)^{-N},
\end{equation}
where $\displaystyle D^{\bf k} f(x,t)=\frac{\partial^{k_0}}{\partial t^{k_0}}\frac{\partial^{k_1}}{\partial x_1^{k_1}}\ldots\frac{\partial^{k_d}}{\partial x_d^{k_d}}f(x,t)$.

\subsubsection*{The proof of Theorem \ref{t_1}.}%\label{s_sec2}

 For the sake of brevity we use the following notational conventions
\begin{equation}\label{nota_conve}
\begin{array}{c}
\partial_{z_j}=\frac{\partial}{\partial z_j}, \quad \partial_t=\frac{\partial}{\partial t}, \\[3mm]
(\partial_{x_j}f)\big(\frac x\eps\big)=
\partial_{z_j}f(z)\big|_{z=x/\eps},\quad(\partial_tf)\big(\frac t{\eps^\alpha}\big)=
\partial_{s}f(s)\big|_{s=t/\eps^\alpha}.
\end{array}
\end{equation}
Denote
$$
\widehat a^{0,ij}(z,s)=a^{ij}(z,s)+a^{im}(z,s)\partial_{z_m}\chi^{0,j}(z,s)+\partial_{z_m}\big(a^{mi}(z,s)\chi^{0,j}(z,s)\big),
$$
$$
\widehat a^{k,ij}(z,s)=a^{im}(z,s))\partial_{z_m}\chi^{k,j}(z,s)+\partial_{z_m}\big(a^{mi}(z,s)\chi^{k,j}(z,s)\big),\quad k=1,\,2,\ldots,
$$
and from \eqref{eq:k_effec_matrix}
$$a^{k,{\rm eff}}=\mathbf{E}\int\limits_{\mathbb T^d}
\big[\widehat a^k(z,s)\big]dz,\quad k=1,\,2,\ldots
$$
Substituting $V^\eps$ for $u^\eps$ in \eqref{ori_cauch} yields
\begin{equation}\label{prob_V}
\begin{array}{cc}
\partial_t V^\eps-
\mathrm{div}
\big[a\big(\frac x\eps,\frac t{\eps^\alpha}\big)\nabla V^\eps\big]=-\eps^{-\frac\alpha2}\sum\limits_{k=0}^{J^0}\eps^{k\delta}\Big[\partial_tu^k
\\[2.5mm]
+\sum\limits_{j=0}^{J^0-k}\eps^{(j\delta+1-\alpha)}\big(\partial_t\chi^j\big)\big(\frac x\eps,\frac t{\eps^\alpha}\big)\nabla u^k
+\sum\limits_{j=0}^{J^0-k}\eps^{(j\delta+1)}\chi^j\big(\frac x\eps,\frac t{\eps^\alpha}\big)\partial_t\nabla u^k
\Big]\\[3mm]
+\eps^{-\frac\alpha2}\sum\limits_{k=0}^{J^0}\eps^{k\delta-1}\Big[(\mathrm{div} a)\big(\frac x\eps,\frac t{\eps^\alpha}\big)
+\sum\limits_{j=0}^{J^0-k}\eps^{j\delta}\big(\mathrm{div}(a\nabla\chi^j)\big)\big(\frac x\eps,\frac t{\eps^\alpha}\big)\Big]
\nabla u^k\\[3mm]
+\eps^{-\frac\alpha2}\sum\limits_{k=0}^{J^0}\sum\limits_{j=0}^{J^0-k}\eps^{(k+j)\delta}\ \widehat a^{j,im}\big(\frac x\eps,\frac t{\eps^\alpha}\big)
\frac{\partial^2}{\partial x_i\partial x_m}u^k\\[3mm]
+\eps^{-\frac\alpha2}\sum\limits_{k=0}^{J^0}\sum\limits_{j=0}^{J^0-k}\eps^{(k+j)\delta+1}\ (a^{im}\chi^{j,l})\big(\frac x\eps,\frac t{\eps^\alpha}\big)
\frac{\partial^3}{\partial x_i\partial x_m\partial x_l}u^k,\\[4.5mm]
V^\eps(x,0)=\sum\limits_{j=0}^{J^0}\eps^{(j\delta+1)}\chi^j\Big(\frac x\eps,0\Big)\nabla u^0(x,0)
\end{array}
\end{equation}
Due to \eqref{aux_0} and \eqref{aux_j},
\begin{equation*}\label{prob_V_ext100000}
\begin{array}{l}
-\sum\limits_{k=0}^{J^0}\eps^{k\delta}
\sum\limits_{j=0}^{J^0-k}\eps^{(j\delta+1-\alpha)}\big(\partial_t\chi^j\big)\big(\frac x\eps,\frac t{\eps^\alpha}\big)\nabla u^k\\[3mm]
\quad +\sum\limits_{k=0}^{J^0}\eps^{k\delta-1}\Big[(\mathrm{div} a)\big(\frac x\eps,\frac t{\eps^\alpha}\big)
+\sum\limits_{j=0}^{J^0-k}\eps^{j\delta}\big(\mathrm{div}(a\nabla\chi^j)\big)\big(\frac x\eps,\frac t{\eps^\alpha}\big)\Big]
\nabla u^k \\[3mm]
=-\eps^{(J^0+1)\delta-1}\sum\limits_{k=0}^{J^0}\big(\partial_t\chi^{J^0-k}\big)\big(\frac x\eps,\frac t{\eps^\alpha}\big)\nabla u^k.
\end{array}
\end{equation*}
Considering our choice of $J^0$ we have: $(J^0+1)\delta-1 > 1+\al/2$.
Therefore, with the help of  (\ref{eff_cauch}) and  (\ref{eff_cauch_j}) the first relation in \eqref{prob_V}
can be rearranged as follows
 \begin{equation}\label{prob_V_ext1}
\begin{array}{cc}
\partial_t V^\eps-
\mathrm{div}
\big[a\big(\frac x\eps,\frac t{\eps^\alpha}\big)\nabla V^\eps\big]=
\\[2.5mm]-\eps^{-\frac\alpha2}\sum\limits_{k=0}^{J^0}\eps^{k\delta}\partial_tu^k
+\eps^{-\frac\alpha2}\sum\limits_{k=0}^{J^0}\sum\limits_{j=0}^{J^0-k}\eps^{(k+j)\delta}\ \widehat a^{j,im}\big(\frac x\eps,\frac t{\eps^\alpha}\big)
\frac{\partial^2 u^k}{\partial x_i\partial x_m} +\mathcal{R}^\eps(x,t)\\[3mm]
=\eps^{-\frac\alpha2}\sum\limits_{j=0}^{J^0}\sum\limits_{k=0}^{J^0-j}\eps^{(k+j)\delta}\big[\widehat a^j\big(\frac x\eps,\frac t{\eps^\alpha}\big)-a^{j,{\rm eff}}\big]^{im}
\frac{\partial^2 u^k}{\partial x_i\partial x_m}\ +\mathcal{R}^\eps(x,t),
\end{array}
\end{equation}
where we identify $a^{0,{\rm eff}}$ with $\aeff$, and $\mathcal{R}^\eps$ is the sum of all the terms on the right-hand side in \eqref{prob_V}
that are multiplied by a positive power of $\eps$. One can easily check that
\begin{equation}\label{def_cal_R}
\mathcal{R}^\eps(x,t)=\eps^{-\alpha/2} \sum\limits_{j=0}^{J_0}\eps^{1+j\delta}\theta^j\Big(\frac x\eps,\frac t{\eps^\alpha}\Big)\Phi^j(x,t),
\end{equation}
where  $\theta^j(z,s)$ are periodic in $z$, stationary in $s$ and satisfy the estimates
\begin{equation}\label{est_thet}
\mathbf{E}\big(\|\theta^j\|^k_{C(\mathbb T^d\times[0,T])}\big)\leq C_k;
\end{equation}
$\Phi^j$ are smooth functions such that
\begin{equation}\label{schwa_phi}
|D^{\bf k}\Phi^j| \leq C_{{\bf k},N}(1+|x|)^{-N},
\end{equation}
and $N_0$ is a finite number; we do not specify these quantities explicitly because we do not need
this.
We represent $V^\eps$ as the sum $V^\eps= V_1^\eps+V_2^\eps$, where
$ V_1^\eps$ and $V_2^\eps$ solve the following problems:
\begin{equation}\label{prob_V1}
\left\{\!\!\begin{array}{rl}
\displaystyle
\partial_t V_1^\eps\!-&\!\!\!\!\! \displaystyle\mathrm{div}
\Big[a\Big(\frac x\eps,
\frac t{\eps^\alpha}\Big)\nabla V_1^\eps\Big]\!
\\[4mm]
=\!\!& \!\!\displaystyle
\eps^{-\alpha/2}\sum\limits_{k=0}^{J^0}\sum\limits_{j=0}^{J^0-k}\eps^{(k+j)\delta}\Big[\widehat a^j\Big(\frac x\eps,\frac t{\eps^\alpha}\Big)-a^{j,{\rm eff}}\Big]^{im}
\frac{\partial^2}{\partial x_i\partial x_m}u^k,\\[3mm]
V_1^\eps(x,&\!\!\!\!\!0) =0,
\end{array}\right.
\end{equation}
and
\begin{equation}\label{prob_V2}
\left\{\begin{array}{l}
\displaystyle
\partial_t V_2^\eps-\mathrm{div}\displaystyle
\Big[a\Big(\frac x\eps, \displaystyle
\frac t{\eps^\alpha}\Big)\nabla V_2^\eps\Big]
=\mathcal{R}^\eps(x,t),\\[3mm]
V_2^\eps(x,0)=V^\eps(x,0).
\end{array}\right.
\end{equation}
Form (\ref{est_chi0}) and \eqref{est_chij} it follows that the initial condition in the latter problem
satisfies for any $k>0$ the estimate $\mathbf{E}\|V^\eps(\cdot,0)\|^k_{C(\mathbb R^d)}\leq C_k\eps^{k\delta/2}$.
If we multiply equation \eqref{prob_V2} by  $V_2^\eps$ and integrate the resulting relation over $\mathbb R^d\times(0,T)$, then integrating by parts
and  combining estimates \eqref{def_cal_R}, \eqref{est_thet} and the estimates for $\Phi^j$,  we obtain
\begin{equation}\label{estiV1}
\mathbf{E}\|V_2^\eps\|^2_{L^2(\mathbb R^d\times(0,T))}\leq C\eps^\delta.
\end{equation}

Denote
$$
\langle a\rangle^0(s)=\int_{\mathbb T^d}\big(\widehat a^0(z,s)-\aeff\big)dz
$$
$$
\langle a\rangle^k(s)=\int_{\mathbb T^d}\big(\widehat a^k(z,s)-a^{k,{\rm eff}}\big)dz,
\quad k=1,\,2,\ldots
$$
It follows from the definition of $\widehat a^k$ that for any $k\geq0$ and $l>0$ there is a constant $C_{l,k}$ such that
$\mathbf{E}\|(\widehat a^k-\langle a\rangle^k)\|^N_{C^k(\mathbb T^d\times[0,T])}\leq C_{N,k}$. Since for each $s\in\mathbb R$
the mean value of $(\widehat a^k(\cdot,s)-\langle a\rangle^k(s))$ is equal to zero, the problem
$$
\Delta_z\zeta^{k,im}(z,s)=   (\widehat a^k(z,s)-\langle a\rangle^k(s))^{im}
$$
has for each $i$ and $m$ a unique up to an additive constant periodic solution. Letting $\Theta^{k,im}(z,s)=\nabla\zeta^{k,im}(z,s)$, we obtain a stationary in $s$ vector functions $\Theta^{k,im}$ such that
$$
\mathrm{div}\,\Theta^{k,im}(z,s)=(\widehat a^k(z,s)-\langle a\rangle^k(s))^{im}, \qquad {\bf E}\|\Theta^{k,im}\|^N_{C^k(\mathbb T^d\times[0,T])}
\leq C_{N,k}.
$$
It is then  straightforward to check that for the functions
$$
F^{\eps}(x,t)=\eps^{-\alpha/2}\sum\limits_{k=0}^{J^0}\sum\limits_{j=0}^{J^0-k}\eps^{(k+j)\delta}
\Big[\widehat a^j\Big(\frac x\eps,\frac t{\eps^\alpha}\Big)-\langle a\rangle^j\Big(\frac t{\eps^\alpha}\Big)\Big]^{im}
\frac{\partial^2}{\partial x_i\partial x_m}u^k(x,t)
$$
$$
=\eps^{1-\frac\alpha2}\sum\limits_{k=0}^{J^0}\sum\limits_{j=0}^{J^0-k}\eps^{(k+j)\delta}\Big\{
\mathrm{div}\Big[\Theta^{j,im}\Big(\frac x\eps,\frac t{\eps^\alpha}\Big)\frac{\partial^2}{\partial x_i\partial x_m}u^k(x,t)\Big]
$$
$$
\hbox{ }\hskip3.5cm-
\Theta^{j,im}\Big(\frac x\eps,\frac t{\eps^\alpha}\Big)\nabla\Big(\frac{\partial^2}{\partial x_i\partial x_m}u^k(x,t)\Big)\Big\}
$$
the following estimate is fulfilled:
\begin{equation} \label{eq:estimF_eps}
{\bf E}\|F^{\eps}\|^2_{L^2(0,T; H^{-1}(\mathbb R^d))}\leq C\eps^\delta.
\end{equation}
Therefore, a solution to the problem
\begin{equation}\label{prob_V12}
\left\{\begin{array}{l}
\displaystyle
\partial_t V_{1,2}^\eps-\mathrm{div}\displaystyle
\Big[a\Big(\frac x\eps, \displaystyle
\frac t{\eps^\alpha}\Big)\nabla V_{1,2}^\eps\Big]
=F^\eps(x,t),\\[3mm]
V_{1,2}^\eps(x,0)=0.
\end{array}\right.
\end{equation}
admits the estimate
\begin{equation}\label{esti_V12}
\mathbf{E}\|V_{1,2}^\eps\|^2_{L^2(0,T;H^1(\mathbb R^d))}\leq C\eps^\delta.
\end{equation}
Splitting $V^\eps_1=V^\eps_{1,1}+V^\eps_{1,2}$, we conclude that $V^\eps_{1,1}$ solves the following problem
\begin{equation}\label{prob_V11}
\left\{\!\!\begin{array}{rl}
\displaystyle
\partial_t V_{1,1}^\eps\!-&\!\!\!\mathrm{div}\displaystyle
\Big[a\Big(\frac x\eps, \displaystyle
\frac t{\eps^\alpha}\Big)\nabla V_{1,1}^\eps\Big]\! =
\\[4mm]
\!\!\!& \displaystyle
\eps^{-\alpha/2}\sum\limits_{k=0}^{J^0}\sum\limits_{j=0}^{J^0-k}\eps^{(k+j)\delta}\Big[\langle a\rangle^j\Big(\frac t{\eps^\alpha}\Big)-
a^{j,{\rm eff}}\Big]^{im} \frac{\partial^2u^k}{\partial x_i\partial x_m},\\[3mm]
V_{1,1}^\eps(x,&\!\!\!0)=0,
\end{array}\right.
\end{equation}
By construction the strong mixing coefficient of $\widehat a^k$ remains unchanged and is equal to $\gamma(\cdot)$.
Denote by $V_{1,1}^{0,\eps}$ the solution of the following problem
\begin{equation}\label{prob_V110}
\left\{\begin{array}{l}
\displaystyle
\partial_t V_{1,1}^{0,\eps}-\mathrm{div}\displaystyle
\Big[a\Big(\frac x\eps,
\frac t{\eps^\alpha}\Big)\nabla V_{1,1}^{0,\eps}\Big]
=\eps^{-\frac\alpha2}\Big[\langle a\rangle^0\Big(\frac t{\eps^\alpha}\Big)-\aeff\Big]^{im}
\frac{\partial^2u^0}{\partial x_i\partial x_m}\\[4mm]
V_{1,1}^{0,\eps}(x,0)=0,
\end{array}\right.
\end{equation}
\begin{lmm} \label{lmm:weak_conv}
The solution of problem \eqref{prob_V110} converges in law, as $\eps\to0$, in $L^2(\mathbb R^d\times(0,T))$
equipped with strong topology, to the solution of \eqref{eff_spde}.
\end{lmm}
\begin{proof}
We consider an auxiliary problem
\begin{equation}\label{prob_V_auxi}
\left\{\begin{array}{l}
\displaystyle
\partial_t V_{\rm aux}^{\eps}-\mathrm{div}\displaystyle \Big[\aeff \nabla V_{\rm aux}^{\eps}\Big]
=\eps^{-\alpha/2}\Big[\langle a\rangle^0\Big(\frac t{\eps^\alpha}\Big)-\aeff \Big]^{ij}
\frac{\partial^2u^0}{\partial x^i\partial x^j}\\[4mm]
V_{\rm aux}^{\eps}(x,0)=0,
\end{array}\right.
\end{equation}
and notice that this problem admits an explicit solution
$$
V_{\rm aux}^{\eps}=\eps^{\alpha/2}\zeta^{ij}\Big(\frac{t}{\eps^\alpha}\Big)\frac{\partial^2u^0}{\partial x^i\partial x^j}
\qquad\hbox{with }\ \zeta(s)=\int_0^s\big[\langle a\rangle^0(r)-\aeff \big]\,dr.
$$
Due to \cite[Lemma VIII.3.102]{JaShi}, \cite[Theorem  VIII.3.97]{JaShi} and Assumption {\bf c5.} the invariance principle holds for the process
$\eps^{\alpha/2}\zeta^{ij}\big(\frac{t}{\eps^\alpha}\big)$, that is $\eps^{\alpha/2}\zeta^{ij}\big(\frac{t}{\eps^\alpha}\big)$,
converges in law, as $\eps\to0$, in $C([0,T])^{d^2}$ to a $d^2$-dimensional Brownian motion with the covariance matrix $\Lambda$.
Since $u^0$ satisfies estimates \eqref{est_uzero}, the last convergence implies that $V_{\rm aux}^{\eps}$   converges in law in
$C((0,T);L^2(\mathbb R^d))$ to the solution of problem \eqref{eff_spde}.

Next, we represent $ V_{1,1}^{0,\eps}$ as $ V_{1,1}^{0,\eps}(x,t)=\mathcal{Z}^{\eps}(x,t)+V_{\rm aux}^{\eps}(x,t)$. Then $\mathcal{Z}^{\eps}$
solves the problem
\begin{equation}\label{prob_Zeps}
\left\{\begin{array}{l}
\displaystyle
\partial_t \mathcal{Z}^{\eps}-\mathrm{div}\displaystyle
\Big[a\Big(\frac x\eps,
\frac t{\eps^\alpha}\Big)\nabla \mathcal{Z}^{\eps}\Big]
=
\mathrm{div}\Big(\Big[a\Big(\frac x\eps,\frac t{\eps^\alpha}\Big)-\aeff \Big]\nabla
V_{\rm aux}^{\eps}(x,t)\Big)\\[4mm]
\mathcal{Z}^{\eps}(x,0)=0,
\end{array}\right.
\end{equation}
and our goal is to show that $\mathcal{Z}^{\eps}$ goes to zero in probability in $L^2((0,T)\times\mathbb R^ d)$,
as $\eps\to0$. To this end we consider one more auxiliary problem that reads
\begin{equation}\label{prob_Yeps}
\left\{\begin{array}{l}
\displaystyle
\partial_t \mathcal{Y}^{\eps}-\mathrm{div}\displaystyle
\Big[a\Big(\frac x\eps,
\frac t{\eps^\alpha}\Big)\nabla \mathcal{Y}^{\eps}\Big]
=
\mathrm{div}\Big(\Big[a\Big(\frac x\eps,\frac t{\eps^\alpha}\Big)-\aeff \Big]
\Xi(x,t)\Big)\\[4mm]
\mathcal{Y}^{\eps}(x,0)=0.
\end{array}\right.
\end{equation}
If the vector function $\Xi\in L^ 2((0,T)\times\mathbb R^d)$, then this problem has a unique solution, and, by the standard energy estimate,
$$
\| \mathcal{Y}^{\eps}\|_{L^2(0,T;H^1(\mathbb R^d))}+\|\partial_t \mathcal{Y}^{\eps}\|_{L^2(0,T;H^{-1}(\mathbb R^d))}\leq C \| \Xi\|_{L^2((0,T)\times\mathbb R^d)}.
$$
According to \cite[Lemma 1.5.2]{lions2002} the family $\{ \mathcal{Y}^{\eps}\}$ is locally compact in $L^2((0,T)\times\mathbb R^d)$.
Combining this with Aronson's estimate (see \cite{aron:68}) we conclude that the family $\{ \mathcal{Y}^{\eps}\}$ is compact in $L^2((0,T)\times\mathbb R^d)$.\\
Assume for a while that $\Xi$ is smooth and satisfies estimates \eqref{est_uzero}.  Multiplying equation \eqref{prob_Yeps} by a test function of the form
$\varphi(x,t)+\eps \chi^0\big(\frac x\eps,\frac t{\eps^\alpha}\big)\nabla\varphi(x,t)$ with $\varphi\in C_0^\infty((0,T)\times\mathbb R^d)$ and integrating
the resulting relation  yields
$$
\begin{array}{ccc}
  -\int_0^T\int_{\mathbb R^d} \mathcal{Y}^{\eps}\Big(\partial_t\varphi+
  \eps^{1-\alpha}(\partial_t \chi^0)\big(\frac x\eps,\frac t{\eps^\alpha}\big)\nabla\varphi +\eps \chi^0\big(\frac x\eps,\frac t{\eps^\alpha}\big)\partial_t \nabla\varphi(x,t) \Big)
dxdt \\[4mm]
+\int\limits_0^T\int\limits_{\mathbb R^d} \partial_{x_m}\mathcal{Y}^{\eps}a^{im}\big(\frac x\eps,\frac t{\eps^\alpha}\big)
\Big[\partial_{x_i}\varphi+
\big(\partial_{x_i}\chi^{0,j}\big)\big(\frac x\eps,\frac t{\eps^\alpha}\big)\partial_{x_j}\varphi +\eps
\chi^{0,j}\big(\frac x\eps,\frac t{\eps^\alpha}\big)\partial_{x_i}\partial_{x_j}\varphi\Big]dxdt
\\[4mm]
=\int\limits_0^T\!\int\limits_{\mathbb R^d}\big[a\big(\frac x\eps,\frac t{\eps^\alpha}\big)-{\rm a}^{\rm eff}\big]^{im}\Xi^m
\big[\partial_{x_i}\varphi+
\big(\partial_{x_i}\chi^{0,j}\big)\big(\frac x\eps,\frac t{\eps^\alpha}\big)\partial_{x_j}\varphi +\eps
\chi^{0,j}\big(\frac x\eps,\frac t{\eps^\alpha}\big)\frac{\partial^2 \varphi}{\partial x_i\partial x_j}\big]dxdt
\end{array}
$$
Since $\int_{\mathbb T^d}\chi^0(z,s)dz=0$, we have $\|(\partial_t\chi^0)(x/\eps,t/\eps^\alpha)\nabla\varphi\|_{L^2(0,T;H^{-1}(\mathbb R^d))}\leq C\eps$.
Therefore, $\int_0^T\int_{\mathbb R^d} \mathcal{Y}^{\eps}
  \eps^{1-\alpha}(\partial_t \chi^0)\big(\frac x\eps,\frac t{\eps^\alpha}\big)\nabla\varphi
dxdt $ tends to zero, as $\eps\to0$. Considering \eqref{aux_0} and \eqref{def_effema} we obtain
$$
\begin{array}{ccc}
\int\limits_0^T\int\limits_{\mathbb R^d} \partial_{x_m}\mathcal{Y}^{\eps}a^{im}\big(\frac x\eps,\frac t{\eps^\alpha}\big)
\Big[\partial_{x_i}\varphi+\big(\partial_{x_i}
\chi^{0,j}\big)\big(\frac x\eps,\frac t{\eps^\alpha}\big)\partial_{x_j}\varphi \Big]dxdt
\\[3mm]
=
-\int\limits_0^T\int\limits_{\mathbb R^d}\mathcal{Y}^{\eps}\big\{a\big(\frac x\eps,\frac t{\eps^\alpha}\big)
\big[{\bf I}+\big(\nabla\chi^0\big)\big(\frac x\eps,\frac t{\eps^\alpha}\big) \big]\big\}^{ij}\frac{\partial^2\varphi}{\partial x_i\partial x_j} dxdt
\end{array}
$$
and
$$
\begin{array}{c}
\int\limits_0^T\!\int\limits_{\mathbb R^d}\big[a\big(\frac x\eps,\frac t{\eps^\alpha}\big)-{\rm a}^{\rm eff}\big]^{im}\Xi^m
\big[\partial_{x_i}\varphi+\big(\partial_{x_i}\chi^{0,j}\big)\big(\frac x\eps,\frac t{\eps^\alpha}\big)\partial_{x_j}\varphi +\eps
\chi^{0,j}\big(\frac x\eps,\frac t{\eps^\alpha}\big)\frac{\partial^2\varphi}{\partial x_i\partial x_j}\big]dxdt\\
=
\int\limits_0^T\!\int\limits_{\mathbb R^d}\big\{a\big(\frac x\eps,\frac t{\eps^\alpha}\big)\big[{\bf I}+
\big(\nabla\chi^0\big)\big(\frac x\eps,\frac t{\eps^\alpha}\big) \big]-{\rm a}^{\rm eff}\big\}^{im}\Xi^m
\partial_{x_i}\varphi dxdt\\
-\int\limits_0^T\!\int\limits_{\mathbb R^d}\{{\rm a}^{\rm eff}\}^{im}\Xi^m\big(\partial_{x_i}\chi^{0,j}\big)\big(\frac x\eps,\frac t{\eps^\alpha}\big)\partial_{x_j}\varphi dxdt\\
+
\int\limits_0^T\!\int\limits_{\mathbb R^d}\big[a\big(\frac x\eps,\frac t{\eps^\alpha}\big)-{\rm a}^{\rm eff}\big]^{im}\Xi^m
\eps \chi^{0,j}\big(\frac x\eps,\frac t{\eps^\alpha}\big)\frac{\partial^2\varphi}{\partial x_i\partial x_j}dxdt
\to 0,
\end{array}
$$
as $\eps\to0$.  Denoting by $\mathcal{Y}^{0}$ the limit of $\mathcal{Y}^{\eps}$ for a subsequence, we conclude that
$$
 \int_0^T\int_{\mathbb R^d} \mathcal{Y}^{0}\Big(-\partial_t\varphi -(\aeff)^{ij}\frac{\partial^ 2\varphi}{\partial x_i\partial x_j} \Big)dxdt=0.
$$
Therefore, $ \mathcal{Y}^{0}=0$, and the whole family $\mathcal{Y}^{\eps}$ a.s. converges to $0$ in $L^2((0,T)\times\mathbb R^d)$.
By the density argument this convergence also holds for any $\Xi\in L^2((0,T)\times\mathbb R^d)$.
Since $V_{\rm aux}^{\eps}$   converges in law in $C((0,T);L^2(\mathbb R^d))$, the solution of problem \eqref{prob_Zeps} converges to
zero in probability in  $L^2((0,T)\times\mathbb R^d)$,  and the statement of the lemma follows.
\end{proof}

From the last lemma it follows that the solution of  problem
\eqref{prob_V11} converges in law, as $\eps\to0$, in $L^2(\mathbb R^d\times(0,T))$
equipped with strong topology, to the solution of (\ref{eff_spde}).
Combining this convergence with \eqref{estiV1} and \eqref{esti_V12}
we conclude that $V^\eps$ converges in law in the same space  to the solution of (\ref{eff_spde}).
 This completes the proof of Theorem \ref{t_1}.

\section{Diffusion case }\label{s_sec4}
%-----------------

In this second setting we assume that the matrix
$a(z,s)$ has the form
\begin{equation} \label{eq:a_diff_case}
a(z,s)={\rm a}(z,\xi_s),
\end{equation}
where  $\xi_s$ is a stationary diffusion process in $\mathbb R^n$ with a generator
$$
\mathcal{L}= \frac{1}{2} \trace [q(y)D^2]+ b(y).\nabla
$$
($\nabla$ stands for the gradient, $D^2$ for the Hessian matrix). In this case even for smooth functions ${\rm a}(z,y)$ the coefficients  of matrix ${\rm a}(z,\xi_s)$ are just H\"older continuous in and not differentiable in time, and the method used in the smooth case fails to work.

\medskip

We still assume that Conditions {\bf a1} and {\bf a2} hold. Moreover we suppose that the matrix-functions ${\rm a}(z,y)$, $q(y)$ and vector-function $b(y)$ possess the following properties:
\medskip\noindent
\begin{itemize}
\item [\bf c1.]  ${\rm a}={\rm a}(z,y)$ is periodic in $z$ and smooth in both variables $z$ and $y$. Moreover, for each $N>0$ there exists $C_N>0$ such that
$$
\|{\rm a}\|_{C^N(\mathbb T^d\times\mathbb R^n)}\leq C_N.
$$
\item [\bf c2.] The matrix $q=q(y)$  satisfies the uniform ellipticity conditions: there exist $\lambda>0$ such that
$$
\lambda^{-1}|\zeta|^2\leq q(y)\zeta\cdot\zeta\leq \lambda|\zeta|^2, \quad y,\,\zeta\in\mathbb R^n.
$$
Moreover there exists a matrix $\sigma=\sigma(y)$ such that $q(y) = \sigma^*(y)\sigma(y)$.
\item [\bf c3.] The matrix function $\sigma$ and vector-function $b$ are smooth, for each $N>0$ there exists $C_N>0$ such that
$$
\|\sigma\|_{C^N(\mathbb R^n)}\leq C_N,\qquad \|b\|_{C^N(\mathbb R^n)}\leq C_N.
$$
\item [\bf c4.] The following inequality holds for some $R>0$ and $C_0>0$ and $p>-1$:
 $$
 b(y)\cdot y\leq -C_0|y|^p \quad \hbox{for all } y\in\{y\in\mathbb R^n\,:\,|y|\geq R\}.
 $$
\end{itemize}
We say that Condition {\bf (C)} holds if {\bf a1}, {\bf a2} and {\bf c1} -- {\bf c4} are satisfied. This case is called the {\it diffusive case}.

\subsection{Existing results}

Again according to \cite{KP_1995}, under {\bf (C)},  the sequence $u^\eps$ converges in probability, as $\eps\to0$, to a solution $u^0$ of problem \eqref{eff_cauch}. Corrector $\chi^0=\chi^0(z,y)$ is a periodic solution of the equation
\begin{equation}\label{aux_0_dif}
\mathrm{div}_z\big({\rm a}(z,y)\nabla_z \chi^0(z,y)\big)=-\mathrm{div}_z {\rm a}(z,y);
\end{equation}
here $y\in\mathbb R^n$ is a parameter. We choose an additive constant in such a way that $\int_{\mathbb T^d}\chi^0(z,y)\,dz=0$.
Let us emphasize that it follows from \eqref{aux_0} and \eqref{aux_0_dif} that the zero order correctors $\chi^0$ coincide in both settings: $\chi^0(z,s)=\chi^0(z,\xi_s)$. The effective matrix is again given by \eqref{def_effema}:
$$
\aeff={\bf E}\int_{\mathbb T^d}\big({\bf I}+{\rm a}(z,\xi_s)\big)\nabla_z\chi^0(z,\xi_s)\,dz.
$$
Let us recall that according to \cite{PP_1} under conditions {\bf c2} and {\bf c4} a diffusion process $\xi_\cdot$ with the generator $\mathcal{L}$ has an invariant measure in $\mathbb R^n$ that has a smooth density $\rho=\rho(y)$. For any $N>0$ it holds
$$(1+|y|)^N \rho(y)\leq C_N$$
with some constant $C_N$. The function $\rho$ is the unique up to a multiplicative constant bounded solution of the equation $\mathcal{L}^*\rho=0$; here $*$ denotes the formally adjoint operator. We assume that the process $\xi_t$ is stationary and distributed with the density $\rho$. The effective matrix can be written here as follows:
$$
\aeff=\int_{\mathbb R^n}\int_{\mathbb T^d} \Big({\rm a}(z,y)+{\rm a}(z,y)\nabla_z \chi^0(z,y)\Big)\rho(y)\,dzdy.
$$

\medskip

In \cite{KPP22}, under the condition that $d=1$, a result similar to Theorem \ref{t_1} is proved. We formulate this result under the assumption that condition {\bf (C)} is fulfilled. As before we introduce several correctors and auxiliary quantities.

Higher order correctors are defined as periodic solutions of the equations
\begin{equation}\label{aux_j_dif}
\mathrm{div}_z\big({\rm a}(z,y)\nabla_z \chi^j(z,y)\big)=-\mathcal{L}_y\chi^{j-1}(z,y), \quad j=1,\,2,\ldots,J^0.
\end{equation}
Notice that $\int_{\mathbb T^d}\chi^{j-1}(z,y)\,dz=0$ for all  $j=1,\,2,\ldots,J^0$, thus the compatibility condition is satisfied and the equations are solvable.

\begin{remark} {\rm We have already mentioned that according to  \eqref{aux_0} and \eqref{aux_0_dif} the zero order correctors coincide in both studied cases. It is interesting to compare the correctors defined in \eqref{aux_j_dif} with the ones given by \eqref{aux_j} and to observe that the higher order correctors need not coincide.  }
\end{remark}

We introduce the matrices
$$
a^{k,{\rm eff}}=\int_{\mathbb R^n}\int_{\mathbb T^d}
\big[{\rm a}(z,y)\nabla_z\chi^k(z,y)+\nabla_z\big({\rm a}(z,y)\chi^k(z,y)\big)\big]\rho(y)\,dzdy,\quad k=1,\,2,\ldots,
$$
and matrix valued functions
\begin{equation}\label{eq:hat_a_0}
\widehat{\rm a}^0(z,y)={\rm a}(z,y)+{\rm a}(z,y)\nabla_z\chi^0(z,y)+\nabla_z\big({\rm a}(z,y)\chi^0(z,y)\big),
\end{equation}
$$
\widehat{\rm a}^k(z,y)={\rm a}(z,y)\nabla_z\chi^k(z,y)+\nabla_z\big({\rm a}(z,y)\chi^k(z,y)\big),\quad k=1,\,2,\ldots,
$$
\begin{equation}\label{eq:angle_a_0}
\langle {\rm a}\rangle^0(y)=\int_{\mathbb T^d}\big(\widehat {\rm a}^0(z,y)-\aeff\big)dz,
\end{equation}
$$
\langle {\rm a}\rangle^k(y)=\int_{\mathbb T^d}\big(\widehat {\rm a}^k(z,y)-a^{k,{\rm eff}}\big)dz,
\quad k=1,\,2,\ldots
$$
The functions $u^j=u^j(x,t)$ are defined as solutions of problems
\begin{equation}\label{eff_cauch_j_diff}
\begin{array}{c}
\displaystyle
\frac\partial{\partial t}u^j=\mathrm{div}(\aeff\nabla u^j)+\sum\limits_{k=1}^{j}
\{a^{k,{\rm eff}}\}^{im} \frac{\partial^2}{\partial x_i \partial x_m}u^{j-k}
\\[3mm]
u^j(x,0)=0
\end{array}
\end{equation}
Since for each $j=1,2,\ldots$ problem (\ref{eff_cauch_j_diff}) has a unique solution, the functions $u^j$ are uniquely defined. Finally, we consider the equation
\begin{equation}\label{eqdefQ0}
\mathcal{L}Q^0(y)=\langle {\rm a}\rangle^0(y).
\end{equation}
According to \cite[Theorems 1 and 2]{pard:vere:01}, this equation has a unique up to an additive constant solution of at most polynomial growth.
Denote
\begin{equation}\label{defLamdif}
\Lambda=\{\Lambda^{ijml}\}=\int_{\mathbb R^n}\Big[\frac{\partial}{\partial y_{r_1}}(Q^{0})^{ij}(y)\Big]q^{r_1r_2}(y)
\Big[\frac{\partial}{\partial y_{r_2}}(Q^{0})^{ml}(y)\Big] \rho(y)\,dy.
\end{equation}
The matrix $\Lambda$ is non-negative. Consequently its square root $\Lambda^{1/2}$ is well defined.

In the diffusive case  the following result holds:
\begin{thm}\label{t_2diff}
Under Condition {\bf (C)}, if $d=1$ or if $\alpha \leq 1$, the normalized functions
$$
U^\eps=\eps^{-\alpha/2}\Big(u^\eps-u^0-\sum\limits_{j=1}^{J_0}\eps^{j(2-\alpha)}u^j\Big)
$$
 converge in law, as $\eps\to0$, in $L^2(\mathbb R^d\times(0,T))$ to the unique solution of \eqref{eff_spde}
with the standard $d^2$-dimensional Brownian motion $W$ and $\Lambda$ defined in \eqref{defLamdif}.
\end{thm}
Note that Remark \ref{r_regu_ini} on $\varphi$ still applies in this case. Let us again emphasize that the case $d=1$ is proved in \cite{KPP22}. Hence only the case $\alpha \leq 1$ is addressed here.

\subsection{Proof of Theorem \ref{t_2diff} for $\alpha \leq 1$}

The beginning is the same as in Section \ref{s_sec3} and is developed in \cite[Section 3.1]{KPP22}. We consider the following expression:
$$
V^\eps(x,t)=\eps^{-\alpha/2}\Big\{u^\eps(x,t)- \sum\limits_{k=0}^{J^0}\eps^{k\delta}\Big(u^k(x,t)+\sum\limits_{j=0}^{J^0-k}\eps^{(j\delta+1)}\chi^j\Big(\frac x\eps,\xi_\frac t{\eps^\alpha}\Big)\nabla u^k(x,t)\Big)\Big\},
$$
where $\chi^j(z,y)$ and $u^k(x,t)$ are defined in \eqref{aux_j_dif} and \eqref{eff_cauch_j_diff}, respectively. We substitute $V^\eps$ for $u^\eps$ in \eqref{ori_cauch} using It\^o's formula:
\begin{align*}
& d V^\eps-
\mathrm{div} \big[a\big(\frac x\eps,\xi_{\frac t{\eps^\alpha}}\big)\nabla V^\eps\big]dt\\
& =-\eps^{-\frac\alpha2}\sum\limits_{k=0}^{J^0}\eps^{k\delta}\Big[\partial_tu^k +\sum\limits_{j=0}^{J^0-k}\eps^{(j\delta+1-\alpha)}\big(\cL_y \chi^j\big)\big(\frac x\eps,\xi_{\frac t{\eps^\alpha}}\big)\nabla u^k \\
&\qquad \qquad  +\sum\limits_{j=0}^{J^0-k}\eps^{(j\delta+1)}\chi^j\big(\frac x\eps,\xi_{\frac t{\eps^\alpha}}\big)\partial_t\nabla u^k
\Big] dt \\
&+
\sum\limits_{k=0}^{J^0}\sum\limits_{j=0}^{J^0-k}\eps^{(1-\alpha+(k+j)\delta)}\sigma(\xi_\frac t{\eps^\alpha})\nabla_y\chi^j\Big(\frac x\eps,\xi_\frac t{\eps^\alpha}\Big)\nabla u^k \,dB_t
\end{align*}
\begin{align*}
&+\eps^{-\frac\alpha2}\sum\limits_{k=0}^{J^0}\eps^{k\delta-1}\Big[(\mathrm{div} a)\big(\frac x\eps,\xi_{\frac t{\eps^\alpha}}\big)
+\sum\limits_{j=0}^{J^0-k}\eps^{j\delta}\big(\mathrm{div}(a\nabla\chi^j)\big)\big(\frac x\eps,\xi_{\frac t{\eps^\alpha}}\big)\Big]
\nabla u^k dt\\
&+\eps^{-\frac\alpha2}\sum\limits_{k=0}^{J^0}\sum\limits_{j=0}^{J^0-k}\eps^{(k+j)\delta}\ \widehat a^{j,im}\big(\frac x\eps,\xi_{\frac t{\eps^\alpha}}\big)
\frac{\partial^2}{\partial x_i\partial x_m}u^k dt\\
&+\eps^{-\frac\alpha2}\sum\limits_{k=0}^{J^0}\sum\limits_{j=0}^{J^0-k}\eps^{(k+j)\delta+1}\ (a^{im}\chi^{j,l})\big(\frac x\eps,\xi_{\frac t{\eps^\alpha}}\big)
\frac{\partial^3}{\partial x_i\partial x_m\partial x_l}u^k dt.
\end{align*}
Here the $n\times n$ matrix $\sigma(y)$ is such that $\sigma(y)\sigma^*(y)=2q(y)$, $B_.$ is a standard $n$-dimensional Brownian motion. Due to \eqref{aux_0_dif} and \eqref{aux_j_dif}
\begin{equation*}\label{prob_V_ext100000_dif}
\begin{array}{l}
-\sum\limits_{k=0}^{J^0}\eps^{k\delta}
\sum\limits_{j=0}^{J^0-k}\eps^{(j\delta+1-\alpha)}\big(\cL_y\chi^j\big)\big(\frac x\eps,\xi_{\frac t{\eps^\alpha}}\big)\nabla u^k\\[3mm]
\quad +\sum\limits_{k=0}^{J^0}\eps^{k\delta-1}\Big[(\mathrm{div} a)\big(\frac x\eps,\xi_{\frac t{\eps^\alpha}}\big)
+\sum\limits_{j=0}^{J^0-k}\eps^{j\delta}\big(\mathrm{div}(a\nabla\chi^j)\big)\big(\frac x\eps,\xi_{\frac t{\eps^\alpha}}\big)\Big]
\nabla u^k \\[3mm]
=-\eps^{(J^0+1)\delta-1}\sum\limits_{k=0}^{J^0}\big(\cL_y\chi^{J^0-k}\big)\big(\frac x\eps,\xi_{\frac t{\eps^\alpha}}\big)\nabla u^k.
\end{array}
\end{equation*}
Considering equations \eqref{eff_cauch_j_diff} and the definitions of $a^{k,{\rm eff}}$ and $\widehat{\rm a}^k(z,y)$,  we obtain an expression similar to that in \eqref{prob_V_ext1}
\begin{align}\label{prob_V_dif_2}
& d V^\eps \displaystyle(x,t)-\mathrm{div}
\Big[{\rm a}\Big(\frac x\eps,\xi_\frac t{\eps^\alpha}\Big)\nabla V^\eps\Big]\,dt\\ \nonumber
& =
\Big(\eps^{-\alpha/2}\sum\limits_{j=0}^{J^0}\sum\limits_{k=0}^{J^0-j}\eps^{(k+j)\delta}\Big[\widehat {\rm a}^k\Big(\frac x\eps,
\xi_\frac t{\eps^\alpha}\Big)-a^{k,{\rm eff}}\Big]^{im} \frac{\partial^2 u^j}{\partial x_i\partial x_m}\Big)\,dt\\ \nonumber
& \quad +
\sum\limits_{k=0}^{J^0}\sum\limits_{j=0}^{J^0-k}\eps^{(1-\alpha+(k+j)\delta)}\sigma(\xi_\frac t{\eps^\alpha})\nabla_y\chi^j\Big(\frac x\eps,\xi_\frac t{\eps^\alpha}\Big)\nabla u^k(x,t)\,dB_t\\ \nonumber
&\quad  +\mathcal{R}^\eps(x,t)\,dt,
 \end{align}
with $a^{0,{\rm eff}} = \aeff$ and the initial condition
$$V^\eps(x,0)=\eps^{1-\alpha/2} \sum\limits_{k=0}^{J^0}\sum\limits_{j=0}^{J^0-k}\eps^{j\delta}\chi^j\Big(\frac x\eps,\xi_0\Big)\nabla u^k(x,0)$$
and
\begin{equation}\label{def_cal_R_dif}
\mathcal{R}^\eps(x,t)=\eps^{-\alpha/2} \sum\limits_{j=0}^{J^0}\eps^{1+j\delta}\vartheta^j\Big(\frac x\eps,\xi_\frac t{\eps^\alpha}\Big)\Phi^j(x,t)
\end{equation}
with periodic in $z$ smooth functions $\vartheta^j=\vartheta^j(z,y)$ of at most polynomial growth in $y$, and $\Phi^j$ satisfying
\eqref{schwa_phi}.

We represent $V^\eps$ as the sum $V^\eps=V_1^\eps+V_2^\eps+V_3^\eps$ where $V^\eps_1$ and $V^\eps_2$ solve problems equivalent to \eqref{prob_V1} and \eqref{prob_V2}:
\begin{equation}\label{prob_V1_dif}
\left\{\!\!\begin{array}{rl}
\displaystyle
\partial_t V_1^\eps\!-&\!\!\!\!\! \displaystyle\mathrm{div}
\Big[{\rm a}\Big(\frac x\eps,
\xi_{\frac t{\eps^\alpha}}\Big)\nabla V_1^\eps\Big]\!
\\[4mm]
=\!\!& \!\!\displaystyle
\eps^{-\alpha/2}\sum\limits_{k=0}^{J^0}\sum\limits_{j=0}^{J^0-k}\eps^{(k+j)\delta}\Big[\widehat {\rm a}^j\Big(\frac x\eps,\frac t{\eps^\alpha}\Big)-a^{j,{\rm eff}}\Big]^{im}
\frac{\partial^2 u^k}{\partial x_i\partial x_m},\\[3mm]
V_1^\eps(x,&\!\!\!\!\!0) =0,
\end{array}\right.
\end{equation}
and
\begin{equation}\label{prob_V2_diff}
\left\{\begin{array}{l}
\displaystyle
\partial_t V_2^\eps-\mathrm{div}\displaystyle
\Big[{\rm a}\Big(\frac x\eps, \displaystyle
\xi_{\frac t{\eps^\alpha}}\Big)\nabla V_2^\eps\Big]
=\mathcal{R}^\eps(x,t),\\[3mm]
V_2^\eps(x,0)=V^\eps(x,0).
\end{array}\right.
\end{equation}
We have
\begin{eqnarray*}
{\bf E}\|\mathcal{R}^\eps\|^2_{L^2(\mathbb R^d\times(0,T))} & \leq & C\eps^{1-\alpha/2}\int_0^T\int_{\mathbb R^d}\int_{\mathbb R^n}
(1+|y|)^{N_1}(1+|x|)^{-2n}\rho(y)\,dydxdt \\
& \leq & C \eps^{1-\alpha/2}.
\end{eqnarray*}
Similarly, ${\bf E}\|V_2^\eps(\cdot,0)\|^2_{L^2(\mathbb R^d)}\leq C\eps^{1-\alpha/2}$. Therefore, $V^\eps_2$ still satisfies \eqref{estiV1} and thus does not contribute in the limit.

The last term $V_3^\eps$ satisfies the SPDE:
\begin{equation}\label{prob_VF_dif}
\begin{array}{l}
\displaystyle
d V_3^\eps\displaystyle(x,t)-\mathrm{div}
\Big[{\rm a}\Big(\frac x\eps,\xi_\frac t{\eps^\alpha}\Big)\nabla V_3^\eps\Big]\,dt\\[4mm]
\hskip 1cm = \displaystyle \eps^{1-\alpha}
\sum\limits_{k=0}^{J^0}\sum\limits_{j=0}^{J^0-k}\eps^{(k+j)\delta}\sigma(\xi_\frac t{\eps^\alpha})\nabla_y\chi^j\Big(\frac x\eps,\xi_\frac t{\eps^\alpha}\Big)\nabla u^k(x,t)\,dB_t
\end{array}
\end{equation}
with initial condition $V_3^\eps(x,0)=0$. Let us again emphasize that the diffusive case cannot be deduced from our first case because of the presence of $V_3^\eps$.

\medskip

We turn to  $V^\eps_1$.  The statement similar to that  of Lemma \ref{lmm:weak_conv} still holds. Indeed the equivalent of $F^\eps$
$$
H^\eps(x,t)=\eps^{-\alpha/2}\sum\limits_{j=0}^{J^0}\sum\limits_{k=0}^{J^0-j}\eps^{(k+j)\delta}\Big[\widehat {\rm a}^k\Big(\frac x\eps,
\xi_\frac t{\eps^\alpha}\Big)-\langle{\rm a}\rangle^k\big(\xi_\frac t{\eps^\alpha}\big)\Big]^{im}
\frac{\partial^2 u^j}{\partial x_i \partial x_m}
$$
admits the estimate \eqref{eq:estimF_eps}:
\begin{equation}\label{est_HHH}
{\bf E}\|H^\eps\|\big.^2_{L^2(0,T;H^{-1}(\mathbb R^d))}\leq C\eps^{2-\alpha}.
\end{equation}
We split $V^\eps_1=V^\eps_{1,1}+V^\eps_{1,2}$, where
\begin{itemize}
\item $V_{1,2}^\eps$ solves \eqref{prob_V12} with $H^\eps$ on the right-hand side instead of $F^\eps$, it admits estimate \eqref{esti_V12};
\item $V^\eps_{1,1}$ solves \eqref{prob_V11}.
\end{itemize}
According to \cite[Theorem 3]{pard:vere:01} the processes
$$
A^k(t)=\int_0^t (\langle a\rangle^k(\xi_s)- a^{k,{\rm eff}} )ds
$$
satisfy the functional central limit theorem (invariance principle), that is the process
$$
A^{\eps,k}(t)=\eps^{\frac\alpha2}\int_0^{\eps^{-\alpha}t} (\langle a\rangle^k(\xi_s)- a^{k,{\rm eff}} )ds
$$
converges in law in $C([0,T]; \R^{d^2})$ to a $d^2$-dimensional Brownian motion with covariance matrix
$$
(\Lambda_k)=\{(\Lambda_k)^{ijml}\}=\int_{\mathbb R^n}\Big[\frac{\partial}{\partial y_{r_1}}(Q^{k})^{ij}(y)\Big]q^{r_1r_2}(y)
\Big[\frac{\partial}{\partial y_{r_2}}(Q^{k})^{ml}(y)\Big] \rho(y)\,dy.
$$
with matrix-function $Q^0$ defined in \eqref{eqdefQ0} and $Q^k$ given by
\begin{equation}\label{eqdefQj}
\mathcal{L}Q^k(y)=\langle {\rm a}\rangle^k(y), \qquad k=1,\ldots.
\end{equation}
By the same arguments as those in the proof of Theorem \ref{t_1} (see also \cite[Lemma 5.1]{KPP_2015}), we obtain the  same conclusions as in  Lemma \ref{lmm:weak_conv}.

\medskip

To finish the proof of Theorem \ref{t_2diff}, we need to control $V^\eps_3$, solution of problem \eqref{prob_VF_dif}. Here we distinguish two cases: $\alpha < 1$ and $\alpha =1$.
As remarked in \cite[Section 4.3]{KPP22}, if $\alpha < 1$, ${\bf E}\| \sup_{0\leq t \leq T} V_3^\eps(\cdot,t)\|^2_{L^2(\mathbb R^d)}\leq C\eps^{1-\alpha}$ and thus this term also does not contribute in the limit equation.  Nonetheless for $\alpha = 1$, the leading term in $V^\eps_3$ solves the SPDE
\begin{equation}
d r^\eps\displaystyle(x,t)-\mathrm{div}
\Big[{\rm a}\Big(\frac x\eps,\xi_\frac t{\eps}\Big)\nabla r^\eps\Big]\,dt  = \displaystyle  \nabla_y\chi^0 \Big(\frac x\eps,\xi_\frac t{\eps}\Big)\nabla u^0(x,t) \sigma(\xi_\frac t{\eps}) \,dB_t.
\end{equation}
\begin{lmm} \label{lmm:alpha=1}
$r^\eps$ converges to zero in probability in
$L^2(0,T;L^2(\R^d)).$
\end{lmm}
Assume for a while that this claim holds. Then due to positive powers of $\eps$ in the other terms of \eqref{prob_VF_dif}, we deduce that $V^\eps_3$ also tends to zero in the same space and the conclusion of Theorem \ref{t_2diff} follows.

\subsection{Proof of Lemma \ref{lmm:alpha=1}}
%-------------------

Let us define
$$v^\eps_t = \int_{\R^d} r^\eps(x,t)^2 dx = \|r^\eps(\cdot,t)\|^2_{L^2(\R^d)}$$
and
$$\Theta^\eps\Big(\frac x\eps,\xi_\frac t{\eps} ,x, t \Big) =  \nabla_y\chi^0 \Big(\frac x\eps,\xi_\frac t{\eps}\Big)\nabla u^0(x,t) \sigma(\xi_\frac t{\eps}).$$
Note that $v^\eps_0 = 0$. It\^o's formula and an integration by part lead to: for any $0 \leq t \leq T$
%\begin{eqnarray*}
%v^\eps_t & = &2 \int_0^t \int_{\R^d} r^\eps(x,s) \mathrm{div} \Big[{\rm a}\Big(\frac x\eps,\xi_\frac s{\eps}\Big)\nabla r^\eps(x,s)\Big] dx ds \\
%& +&2   \int_0^t \int_{\R^d} r^\eps(x,s) \Theta^\eps\Big(\frac x\eps,\xi_\frac s{\eps} ,x, s \Big)dx \,dB_s \\
%& + & \int_0^t \int_{\R^d} \left\| \Theta^\eps\Big(\frac x\eps,\xi_\frac s{\eps} ,x, s \Big)\right\|^2 dx \,ds.
%\end{eqnarray*}
%An integration by part shows that
\begin{eqnarray*}
&& v^\eps_t + 2\int_0^t \int_{\R^d} \nabla r^\eps(x,s)  \Big[{\rm a}\Big(\frac x\eps,\xi_\frac s{\eps}\Big)\nabla r^\eps(x,s)\Big] dx ds \\
&&\quad  =  2 \int_0^t \int_{\R^d} r^\eps(x,s) \Theta^\eps\Big(\frac x\eps,\xi_\frac s{\eps} ,x, s \Big)dx \,dB_s \\
&& \quad +  \int_0^t \int_{\R^d} \left\| \Theta^\eps\Big(\frac x\eps,\xi_\frac s{\eps} ,x, s \Big)\right\|^2 dx \,ds.
\end{eqnarray*}
From Condition {\bf a1}, taking $t=T$ and the expectation, there exists a constant $C$ independent of $\eps$ such that
\begin{equation} \label{eq:H1_boundedness_r_eps}
\mE \int_0^T \left\| \nabla r^\eps(\cdot,s) \right\|^2_{L^2(\R^d)} ds \leq C.
\end{equation}
Moreover by Burkholder-Davis-Gundy, Cauchy-Schwarz and Young inequalities we have
\begin{equation}\label{eq:bound_sup}
\mE \left[ \sup_{t\in[0,T]}  v^\eps_t \right]  = \mE \left[ \sup_{t\in[0,T]} \|r^\eps(\cdot,t)\|^2_{L^2(\R^d)} \right]\leq C.
\end{equation}
Indeed
\begin{align*}
\mE \left[ \sup_{t\in[0,T]}  v^\eps_t \right] & = \mE \left[ \sup_{t\in[0,T]} \|r^\eps(\cdot,t)\|^2_{L^2(\R^d)} \right] \\
& \leq  C \mE \left[   \left(   \int_0^T \left|  \int_{\R^d} r^\eps(x,s) \Theta^\eps\Big(\frac x\eps,\xi_\frac s{\eps} ,x, s \Big)dx \right|^2 \,ds \right)^{1/2} \right]   \\
& + \mE \left[ \int_0^T \int_{\R^d} \left\| \Theta^\eps\Big(\frac x\eps,\xi_\frac s{\eps} ,x, s \Big)\right\|^2 dx \,ds \right]  \\
%&  \leq  C \mE \left[   \left(   \int_0^T v^\eps_s \left(  \int_{\R^d} \left\| \Theta^\eps\Big(\frac x\eps,\xi_\frac s{\eps} ,x, s \Big)\right\|^2 dx \right)  \,ds \right)^{1/2} \right]   \\
%& + \mE \left[ \int_0^T \int_{\R^d} \left\| \Theta^\eps\Big(\frac x\eps,\xi_\frac s{\eps} ,x, s \Big)\right\|^2 dx \,ds \right] \\
&  \leq  \dfrac{1}{2} \mE \left[  \sup_{t\in[0,T]}  v^\eps_t \right] + \dfrac{C}{2} \mE \left[    \int_0^T  \int_{\R^d} \left\| \Theta^\eps\Big(\frac x\eps,\xi_\frac s{\eps} ,x, s \Big)\right\|^2 dx  \,ds \right]   \\
& + \mE \left[ \int_0^T \int_{\R^d} \left\| \Theta^\eps\Big(\frac x\eps,\xi_\frac s{\eps} ,x, s \Big)\right\|^2 dx \,ds \right].
\end{align*}

Now we prove that the sequence $r^\eps$ is tight in
$$V_T= L^2_w(0,T; H^1(\R^d))\cap C(0,T;L^2_w(\R^d)).$$
Remenber that the index $w$ means that the corresponding space is equipped with the weak topology.
We turn to estimating the modulus of continuity of the inner product of $r^\eps$ with a test function $\phi$.

For any function $\phi \in C^\infty_0(\R^d)$ we define
\begin{align*}
\widehat v^\eps_t & = \int_{\R^d} r^\eps(x,t) \left( \phi(x) + \eps \nabla \phi(x) \chi^0 \Big(\frac x\eps,\xi_\frac t{\eps}\Big)  \right) dx  = \langle r^\eps(\cdot,t) ,\phi^\eps(\cdot,t) \rangle_{L^2(\R^d)}.
\end{align*}
Again by It\^o's formula for any $0 \leq t \leq T$
\begin{align*}
\widehat v^\eps_t & =  \int_0^t \int_{\R^d} \phi^\eps(x) \mathrm{div} \Big[{\rm a}\Big(\frac x\eps,\xi_\frac s{\eps}\Big)\nabla r^\eps(x,s)\Big] dx ds \\
%& +  \eps  \int_0^t \int_{\R^d} \nabla \phi(x) \chi^0 \Big(\frac x\eps,\xi_\frac s{\eps}\Big) \mathrm{div} \Big[{\rm a}\Big(\frac x\eps,\xi_\frac s{\eps}\Big)\nabla r^\eps(x,s)\Big] dx ds\\
& +  \int_0^t \int_{\R^d}  r^\eps(x,s)  \nabla \phi(x) \mathcal L_y \chi^0 \Big(\frac x\eps,\xi_\frac s{\eps}\Big)dx ds \\
& + \int_0^t \int_{\R^d} \phi^\eps(x,s) \Theta^\eps\Big(\frac x\eps,\xi_\frac s{\eps} ,x, s \Big)dx \,dB_{s} \\
& +  \eps^{1/2} \int_0^t \int_{\R^d}  r^\eps(x,s)  \nabla \phi(x) \nabla \chi^0 \Big(\frac x\eps,\xi_\frac s{\eps}\Big)dx dB_{s}\\
& +   \eps^{1/2} \int_0^t \int_{\R^d}   \Theta^\eps\Big(\frac x\eps,\xi_\frac s{\eps} ,x, s \Big) \nabla \phi(x) \nabla \chi^0 \Big(\frac x\eps,\xi_\frac s{\eps}\Big)dx ds.
\end{align*}
Evoke that $\widehat{\rm a}^0$ is defined by \eqref{eq:hat_a_0}. With an integration by parts we obtain
%\begin{eqnarray*}
%\widehat v^\eps_t & = &\eps^{-1} \int_0^t \int_{\R^d} r^\eps(x,s) (\mathrm{div}_z {\rm a}) \Big(\frac x\eps,\xi_\frac s{\eps}\Big) \nabla \phi(x) dx ds \\
%& +&  \int_0^t \int_{\R^d} r^\eps(x,s) \mbox{Trace} \Big[ D^2 \phi(x) {\rm a} \Big(\frac x\eps,\xi_\frac s{\eps}\Big)\Big]  dx ds \\
%& +& \int_0^t \int_{\R^d} \phi^\eps(x,s) \Theta^\eps\Big(\frac x\eps,\xi_\frac s{\eps} ,x, s \Big)dx \,dB_{s} \\
%& + & \int_0^t \int_{\R^d}  r^\eps(x,s)  \nabla \phi(x) \mathcal L \chi \Big(\frac x\eps,\xi_\frac s{\eps}\Big)dx ds \\
%& + & \eps^{1/2} \int_0^t \int_{\R^d}  r^\eps(x,s)  \nabla \phi(x) \nabla \chi \Big(\frac x\eps,\xi_\frac s{\eps}\Big)dx dB_{s}\\
%& + &  \eps^{1/2} \int_0^t \int_{\R^d}   \Theta^\eps\Big(\frac x\eps,\xi_\frac s{\eps} ,x, s \Big) \nabla \phi(x) \nabla \chi \Big(\frac x\eps,\xi_\frac s{\eps}\Big)dx ds\\
%& - & \eps  \int_0^t \int_{\R^d} \chi \Big(\frac x\eps,\xi_\frac t{\eps}\Big)\nabla^2 \phi(x) {\rm a}\Big(\frac x\eps,\xi_\frac s{\eps}\Big)\nabla r^\eps(x,s) dx ds \\
%& - &   \int_0^t \int_{\R^d} \nabla \phi(x) \nabla_z \chi \Big(\frac x\eps,\xi_\frac s{\eps}\Big) {\rm a}\Big(\frac x\eps,\xi_\frac s{\eps}\Big)\nabla r^\eps(x,s) dx ds
%\end{eqnarray*}
%
%\langle {\rm a}\rangle^0(y)=\int_{\mathbb T^d}\big(\widehat {\rm a}^0(z,y)-\aeff\big)dz,

\begin{align} \label{eq:dyn_hat_v_eps}
\widehat v^\eps_t & = \int_0^t \int_{\R^d} r^\eps(x,s)  \widehat{\rm a}^0 \Big(\frac x\eps,\xi_\frac s{\eps}\Big) \nabla^2 \phi(x)   dx ds \\ \nonumber
%& + &  \int_0^t \int_{\R^d} r^\eps(x,s) \mathrm{div}_z\Big[ \chi \Big(\frac x\eps,\xi_\frac s{\eps}\Big) {\rm a}\Big(\frac x\eps,\xi_\frac s{\eps}\Big) \Big]\nabla^2 \phi(x) dx ds \\
& +  \int_0^t \int_{\R^d}  r^\eps(x,s)  \nabla \phi(x) \mathcal L_y \chi^0 \Big(\frac x\eps,\xi_\frac s{\eps}\Big)dx ds \\ \nonumber
& + \int_0^t \int_{\R^d} \phi(x,s) \Theta^\eps\Big(\frac x\eps,\xi_\frac s{\eps} ,x, s \Big)dx \,dB_{s} \\ \nonumber
& + \eps^{1/2} \int_0^t \int_{\R^d}  r^\eps(x,s)  \nabla \phi(x) \nabla_z \chi^0 \Big(\frac x\eps,\xi_\frac s{\eps}\Big)dx dB_{s}\\ \nonumber
& +   \eps^{1/2} \int_0^t \int_{\R^d}   \Theta^\eps\Big(\frac x\eps,\xi_\frac s{\eps} ,x, s \Big) \nabla \phi(x) \nabla_z \chi^0 \Big(\frac x\eps,\xi_\frac s{\eps}\Big)dx ds\\ \nonumber
& +  \eps  \int_0^t \int_{\R^d} r^\eps(x,s) \chi^0\Big(\frac x\eps,\xi_\frac s{\eps}\Big)\nabla^3 \phi(x) {\rm a}\Big(\frac x\eps,\xi_\frac s{\eps}\Big) dx ds \\ \nonumber
& + \eps \int_0^t \int_{\R^d} \nabla \phi(x,s) \chi^0\Big(\frac x\eps,\xi_\frac s{\eps}  \Big)\Theta^\eps\Big(\frac x\eps,\xi_\frac s{\eps} ,x, s \Big)dx \,dB_{s}
\end{align}
since from the very definition of $\chi^0$, the two terms  of order $\eps^{-1}$
\begin{eqnarray*}
&& \eps^{-1} \int_0^t \int_{\R^d} (\mathrm{div}_z {\rm a}) \Big(\frac x\eps,\xi_\frac s{\eps}\Big) \nabla \phi(x)  r^\eps(x,s)dx ds \\
& + & \eps^{-1}  \int_0^t \int_{\R^d} \nabla \phi(x)  \mathrm{div}_z \Big[{\rm a}\Big(\frac x\eps,\xi_\frac s{\eps}\Big)  \nabla_z \chi^0 \Big(\frac x\eps,\xi_\frac s{\eps}\Big) \Big] r^\eps(x,s) dx ds
\end{eqnarray*}
are equal to zero.

Using BDG inequality and the estimate \eqref{eq:bound_sup} we deduce that there exists $C> 0$ such that for any $0\leq t \leq \tau \leq T$
$$\mE \left[ \sup_{s\in[t,\tau]}  | \widehat v^\eps_s - \widehat v^\eps_t  | \right]  = C \sqrt{\tau-t}+ C \eps^{1/2}.$$
Hence the sequence $\widehat v^\eps$ is tight in $C(0,T;\mathbb R)$, that is $r^\eps$ is tight in $C(0,T;L^2_w(\R^d))$.

\medskip

For any $i=1,\ldots ,n$, since $\langle \mathcal L_y \chi^0\rangle =\langle (\nabla_y \chi^0)^i \rangle= 0$, we can define $\zeta^{0,i}$ such that ${\rm div}_z \zeta^{0,i}  = (\nabla_y \chi^0)^i$ and ${\rm div}_z \widehat \zeta^0 =  \mathcal L_y \chi^0$ and we have
\begin{align*}
&  \int_0^t \int_{\R^d} \phi(x,s) \Theta^\eps\Big(\frac x\eps,\xi_\frac s{\eps} ,x, s \Big)dx \,dB_{s}\\
& = \int_0^t \int_{\R^d} \phi(x,s)  \nabla_y\chi^0 \Big(\frac x\eps,\xi_\frac s{\eps}\Big)\nabla u^0(x,s)dx \sigma(\xi_\frac s{\eps}) \,dB_{s} \\
& =\eps  \int_0^t \int_{\R^d} \phi(x,s) {\rm div } \zeta^0 \Big(\frac x\eps,\xi_\frac s{\eps}\Big)\nabla u^0(x,s)dx \sigma(\xi_\frac s{\eps}) \,dB_{s} \\
& = - \eps  \int_0^t \int_{\R^d} \zeta^0 \Big(\frac x\eps,\xi_\frac s{\eps}\Big) \nabla \left( \phi(x,s)\nabla u^0(x,s) \right) dx \sigma(\xi_\frac s{\eps}) \,dB_{s}
\end{align*}
and
\begin{align*}
&  \int_0^t \int_{\R^d}  r^\eps(x,s)  \nabla \phi(x) \mathcal L_y \chi^0 \Big(\frac x\eps,\xi_\frac s{\eps}\Big)dx ds \\
& =\eps \int_0^t \int_{\R^d}  r^\eps(x,s)  \nabla \phi(x) {\rm div }\widehat \zeta^0 \Big(\frac x\eps,\xi_\frac s{\eps}\Big)dx ds \\
& = - \eps \int_0^t \int_{\R^d}  \nabla (r^\eps(x,s)  \nabla \phi(x)) \widehat \zeta^0 \Big(\frac x\eps,\xi_\frac s{\eps}\Big)dx ds.
\end{align*}
From \eqref{eq:H1_boundedness_r_eps}, these two quantities converge to zero. Therefore every term in \eqref{eq:dyn_hat_v_eps}, except for the first one, converges to zero.

For the first one, we have
\begin{align*}
& \int_0^t \int_{\R^d} r^\eps(x,s)  \widehat{\rm a}^0 \Big(\frac x\eps,\xi_\frac s{\eps}\Big) \nabla^2 \phi(x)   dx ds\\
& = \int_0^t \int_{\R^d} r^\eps(x,s) \aeff \nabla^2 \phi(x)   dx ds +  \int_0^t \int_{\R^d} r^\eps(x,s) \left(  \langle  {\rm a}\rangle^0 (\xi_\frac s{\eps}) - \aeff \right) \nabla^2 \phi(x)   dx ds \\
& +  \int_0^t \int_{\R^d} r^\eps(x,s) \left(  \widehat{\rm a}^0 \Big(\frac x\eps,\xi_\frac s{\eps}\Big) -  \langle  {\rm a}\rangle^0 (\xi_\frac s{\eps})  \right) \nabla^2 \phi(x)   dx ds ,
\end{align*}
where $\langle {\rm a}\rangle^0$ is defined by \eqref{eq:angle_a_0}. Since
$$\left\langle  \widehat{\rm a}^0 \Big(\frac x\eps,\xi_\frac s{\eps}\Big) -  \langle  {\rm a}\rangle^0 (\xi_\frac s{\eps}) \right\rangle = 0,$$
the last part converges to zero. Moreover by definition of $\aeff$, we also obtain the convergence to zero of the penultimate term. Hence \eqref{eq:dyn_hat_v_eps} becomes:
\begin{align*}
\widehat v^\eps_t &  = \int_{\R^d} r^\eps(x,t) \left( \phi(x) + \eps \nabla \phi(x) \chi^0 \Big(\frac x\eps,\xi_\frac t{\eps}\Big)  \right) dx\\
& = \int_0^t \int_{\R^d} r^\eps(x,s) \aeff \nabla^2 \phi(x)   dx ds \\
& +  \int_0^t \int_{\R^d} r^\eps(x,s) \left(  \langle  {\rm a}\rangle^0 (\xi_\frac s{\eps}) - \aeff \right) \nabla^2 \phi(x)   dx ds \\
& +  \int_0^t \int_{\R^d} r^\eps(x,s) \left(  \widehat{\rm a}^0 \Big(\frac x\eps,\xi_\frac s{\eps}\Big) -  \langle  {\rm a}\rangle^0 (\xi_\frac s{\eps})  \right) \nabla^2 \phi(x)   dx ds + O(\eps^{1/2}).
\end{align*}
Here $O(\eps^{1/2})$ stands for functions whose  $L^2(\Omega ; L^\infty(0,T))$ norm is bounded by a constant times $\eps^{1/2}$. On the right-hand side, the last two integrals converge to zero.
Hence we have proved that the sequence $r^\eps$ converges in probability in $L^2(0,T;L^2_w(\R^d)) \cap L^2(0,T;L^2_{loc}(\R^d)) $ to the unique solution $r^0$ of the PDE \eqref{eff_cauch} with initial value zero. Hence $r^0 = 0$.

\bigskip
To finish the proof of Lemma \ref{lmm:alpha=1}, we now show that the convergence holds in $L^2(0,T;L^2(\R^d))$.
Define a non-negative function $\theta_R \in C^\infty(\R^d)$ equal to zero on $\{|x| \leq R\}$ and equal to one on $\{|x| \geq 3R\}$ and such that $\|\nabla \theta_R\| \leq 1/R$. For
$$\theta_R^\eps (x,t) = \theta_R(x) + \eps \nabla \theta_R(x) \chi^0 \Big(\frac x\eps,\xi_\frac t{\eps}\Big)  $$
we have
\begin{align*}
& d (r^\eps\displaystyle(x,t) \theta^\eps_R(x)) -\mathrm{div}
\Big[{\rm a}\Big(\frac x\eps,\xi_\frac t{\eps}\Big)\nabla (r^\eps \theta^\eps_R) \Big]\,dt  \\
& = \theta^\eps_R(x) \nabla_y\chi^0 \Big(\frac x\eps,\xi_\frac t{\eps}\Big)\nabla u^0(x,t) \sigma(\xi_\frac t{\eps}) \,dB_t \\
& - 2 a\Big(\frac x\eps,\xi_\frac s{\eps}\Big) \nabla r^\eps(x,s) \nabla \theta_R(x) dt  \\
& - r^\eps(x,s) \left[ a + a \nabla\chi^0 + {\rm div } (a\chi^0) \right]\Big(\frac x\eps,\xi_\frac s{\eps}\Big)  \nabla^2 \theta_R (x) dt  \\
& + r^\eps(x,s)  \nabla \theta_R(x) \mathcal L_y \chi^0 \Big(\frac x\eps,\xi_\frac s{\eps}\Big) dt \\
& +  \eps^{1/2}   r^\eps(x,s)  \nabla \theta_R(x) \nabla \chi^0 \Big(\frac x\eps,\xi_\frac s{\eps}\Big) dB_{t} \\
& +   \eps^{1/2}   \Theta^\eps\Big(\frac x\eps,\xi_\frac s{\eps} ,x, s \Big) \nabla \theta_R(x) \nabla \chi^0 \Big(\frac x\eps,\xi_\frac s{\eps}\Big)dt \\
& - \eps r^\eps(x,s) a \Big(\frac x\eps,\xi_\frac s{\eps}\Big)\chi^0\Big(\frac x\eps,\xi_\frac s{\eps}\Big) \nabla^3 \theta_R(x) dt.
\end{align*}
If we apply It\^o's formula to
$$v^R(t) =  \| r^\eps\displaystyle(\cdot,t) \theta^\eps_R(\cdot) \|^2_{L^2(\R^d)}$$
then
\begin{align*}
v^R(t)& +2 \int_0^t \int_{\R^d} \nabla (r^\eps(x,s)  \theta^\eps_R(x)) \Big[{\rm a}\Big(\frac x\eps,\xi_\frac t{\eps}\Big)\nabla (r^\eps(x,s)  \theta^\eps_R(x)) \Big]\,dx ds \\
&= 2  \int_0^t\int_{\R^d}  r^\eps(x,s)  \theta^\eps_R(x) \theta^\eps_R(x) \nabla_y\chi^0 \Big(\frac x\eps,\xi_\frac s{\eps}\Big)\nabla u^0(x,s) \sigma(\xi_\frac s{\eps}) dx\,dB_s \\
& -  4  \int_0^t\int_{\R^d} r^\eps(x,s)  \theta^\eps_R(x) a\Big(\frac x\eps,\xi_\frac s{\eps}\Big) \nabla r^\eps(x,s) \nabla \theta_R(x) dxds  \\
& -  2  \int_0^t\int_{\R^d} r^\eps(x,s)  \theta^\eps_R(x)r^\eps(x,s) \left[ a + a \nabla\chi^0 + {\rm div } (a\chi^0) \right]\Big(\frac x\eps,\xi_\frac s{\eps}\Big)  \nabla^2 \theta_R (x) dxds  \\
& +  2  \int_0^t\int_{\R^d} r^\eps(x,s)  \theta^\eps_R(x) r^\eps(x,s)  \nabla \theta_R(x) \mathcal L_y \chi^0 \Big(\frac x\eps,\xi_\frac s{\eps}\Big) dxds \\
& +  \int_0^t\int_{\R^d}  \theta^\eps_R(x)^2 \left\|  \nabla_y\chi^0 \Big(\frac x\eps,\xi_\frac s{\eps}\Big)\nabla u^0(x,s) \sigma(\xi_\frac s{\eps})  \right\|^2 dxds  + O(\eps^{1/2}).
\end{align*}
Since $u^0$ is a Schwartz class function and $\|\nabla \theta^\eps_R\| \leq 1/R$, the expectation of the right-hand side does not exceed $C \left( \dfrac{1}{R} + \sqrt{\eps} \right)$. It implies tightness in $L^2(0,T;L^2(\R^d))$.

\bigskip\noindent
{\bf Acknowledgements.}

\bibliographystyle{plain}
\bibliography{recherchebib}

\end{document}